\documentclass[12pt]{amsart}
\usepackage{geometry}                % See geometry.pdf to learn the layout options. There are lots.
\usepackage{graphicx}
\usepackage{epstopdf}
\usepackage[pagewise, modulo, displaymath, mathlines]{lineno}
\usepackage{amsmath,amssymb,latexsym,pstricks,
mathrsfs,comment,amsthm,graphicx,tikz,enumerate,pgffor,cite,wrapfig,float}
\usepackage{enumitem}
\usetikzlibrary{cd}

\newtheorem{theorem}{Theorem}[section]

\newtheorem{definition}[theorem]{Definition}
\newtheorem{defs}[theorem]{Definitions}

\newtheorem{lemma}[theorem]{Lemma}

\newtheorem{proposition}[theorem]{Proposition}
\newtheorem{remark}[theorem]{Remark}

\newcommand{\nc}{\newcommand}
\nc{\rnc}{\renewcommand}
\nc{\ip}{idempotent}
\nc{\Ip}{Idempotent}
\nc{\sgp}{semigroup}
\nc{\ia}{inverse algebra}
\nc{\is}{inverse semigroup}
\nc{\alg}{algebra}
\nc{\sla}{semilattice}
\nc{\ioi}{if and only if~}
\nc{\rep}{representation}
\nc{\eff}{effective }
\nc{\tr}{transitive }
\nc{\id}{\rm{id}}
\title[Representations of inverse semigroups] {Representations of inverse semigroups in complete atomistic inverse meet-semigroups\\}
\author{D. G. FitzGerald}
\address{ School of Mathematics and Physics, University of Tasmania, Private Bag 37, Hobart 7001, Australia }
\date{\today}
\email{D.FitzGerald@utas.edu.au}
\begin{document}
\dedicatory{For Vivienne}
\maketitle

\begin{abstract}
 
As an appropriate generalisation of the features of the
classical (Schein) theory of representations of inverse semigroups in $\mathscr{I}_{X}$, a theory of representations of inverse semigroups by homomorphisms into complete atomistic inverse $\wedge$-semigroups is developed.  This class of inverse $\wedge$-semigroups, otherwise known as inverse algebras,  includes partial automorphism monoids of entities such as graphs, vector spaces and modules.  A workable theory of decompositions is reached; however complete distributivity is required for results approaching those of the classical case.
\end{abstract}

\section{Inverse semigroups and   representations}
It is important to study mathematical structures as represented by objects of a suitably elaborate kind: it helps us understand and classify them, as witness the importance of linear groups and groups of automorphisms of graphs.  Inverse semigroups generalise both groups and semilattices, and describe partial symmetries just as groups do for total symmetries; they also arise in representation theory of some operator algebras.  Yet our knowledge of inverse semigroup representations is mostly confined to linear representations as studied by Munn, Ponizovski\u{\i} and others (in which the representing object---the codomain of the representation---is merely a regular rather than inverse semigroup) and to partial permutation representations. 
The latter, the theory of representations of inverse semigroups by  injective partial mappings of a set, is well-developed, beginnning with the Wagner-Preston theorem, and fully developed in the work of Boris Schein.  Namely, any effective representation in the symmetric inverse monoid
$\mathscr{I}_{X}$ decomposes to a `sum' of transitive ones, and every
transitive one has an `internal' description in terms of appropriately defined
cosets of closed inverse subsemigroups.  We shall refer to this as the \emph{classical} treatment.  Section IV.4 of Petrich's book \cite{Pe} has the most helpful exposition, and this paper takes it as a model. For a recent work which also streamlines and modernises the original approaches to the internal descriptions, see \cite{LaMaSt}.   

The intent of the present paper is to explore an approach to the decomposition question which will work for other kinds of partial automorphism monoids.  
Such generalisation is no mere `abstractification' of the classical theory, as it is needed to guide the development of more diverse representations in partial automorphism \alg s of entities such as  graphs, vector spaces, and modules.
So we wish to find
appropriate generalizations of the features of the classical theory, and apply
them where possible in other settings; in particular, we need generalisations of the concepts of effectiveness and transitivity.  

Our point of departure is that many of
these partial automorphism monoids (including prototypically $\mathscr{I}_{X}$
itself) are significantly richer in structure as a result of underlying categorical
properties---they are actually \emph{inverse $\wedge$-semigroups }(see Section 2). Thus it is the contention of this paper that the representation question requires taking account of the properties of inverse $\wedge$-semigroups, and identifying those helpful in the decomposition of representations.

Here is another concrete justification for this endeavour.  Consider the familiar Wagner-Preston representation giving, for any inverse semigroup $S$, an injective morphism $\alpha\colon S\rightarrow \mathscr{I}_{X}$ (where $X$ may be taken to be the carrier set of $S$).  
Recall too that $\mathscr{I}_{X}$ has a counterpart, denoted $\mathscr{I}^{\ast}_{X}$, whose elements are bijections between quotient sets of $X$, rather than between subsets.  (This is described in more detail in Section 4.2 of \cite{La}, and there is a small concrete example in the Appendix below.)
It is proved in \cite{FiLe} that $\mathscr{I}^{\ast}_{X}$ is an inverse $\wedge$-semigroup and that there is an injective morphism $\beta\colon S\rightarrow \mathscr{I}^{\ast}_{X}$, so that every inverse $S$ may be embedded in some $\mathscr{I}^{\ast}_{X}$.  Now let the \emph{degree} of $S$, $\text{deg}S$, be defined as the minimum cardinal $\vert X \vert$ such that 
$S$ embeds in $\mathscr{I}_{X}$; and similarly let $\text{deg}^{\ast}S$ be the minimum cardinal $\vert X \vert$ such that 
$S$ embeds in $\mathscr{I}^{\ast}_{X}$.  Since $\mathscr{I}_{X}$ embeds in $\mathscr{I}^{\ast}_{X\cup 0}$ for $0\notin X$,
(shown in \cite{FiLe}), $\text{deg}^{\ast}S\leq \text{deg}S+1$; and since $\mathscr{I}^{\ast}_{X}$ embeds in $\mathscr{I}_{Y}$ where $Y=2^{X}\smallsetminus\{\varnothing, X\}$, we have $\text{deg}S\leq 2^{\text{deg}^{\ast}S}-2$ (this is shown in the Appendix).  Combining these bounds, we see that %whence 
$\log_{2}(\text{deg}S + 2)\leq \text{deg}^{\ast}S\leq \text{deg}S+1$.  Thus  there is the potential for representation of $S$ in 
$\mathscr{I}^{\ast}_{X}$ to be more efficient than in $\mathscr{I}_{X}$ (in the sense of using a smaller set), at least for inverse semigroups with relatively many  idempotent atoms.  Yet we know very little about representations of $S$ in $\mathscr{I}^{\ast}_{X}$!

Since we shall only deal with inverse semigroups we shall abbreviate terminology, and by 
``subsemigroup'' we shall always mean 
``inverse subsemigroup''.  The paper is organised as follows.
We begin by rehearsing some terminology and foundational results, then  consider ways in which representations may be decomposed into simpler kinds. 
 The concepts of 
\eff and \tr representations are described in terms of the structure of a subsemigroup and how it acts on idempotent atoms.  
Four theorems of increasing particularity, depending on extra properties of the ambient inverse $\wedge$-semigroup, provide information on decomposition of a representation. 
 An Appendix expands on some claims and gives some simple  examples which illustrate the choice of definitions.

\section{Inverse $\wedge$-semigroups and their order properties}

An \emph{inverse $\wedge$-semigroup} $A=(A,\cdot,^{-1},\wedge)$ is an inverse semigroup
$(A,\cdot,^{-1})$ in which the natural ordering is a semilattice order, that
is, for all pairs $a,b\in A$ there is a greatest  $x \in A$ such that $x\leq a,b$; 
this greatest such is denoted $a \wedge b$.  Also known as \emph{inverse algebras} as in \cite{La12}, %(which leaves the sensitive writer with an unenviable choice), 
inverse $\wedge$-semigroups were introduced and elucidated by Leech in \cite{LePLMS} and
\cite{LeScottish}, and the reader is referred to those papers for a full
discussion and examples.  Inverse $\wedge$-semigroups constitute a variety, so the class is closed under the taking of products and subobjects (subsets closed under $\cdot, ^{-1}$ and  $\wedge$).  In particular, the local monoids of $A$ (subsemigroups of the form $eAe$ for some $e = e^{2}$) are themselves inverse $\wedge$-semigroups,  called \emph{local  algebras} for short.  

As alluded to in the introduction,   $\mathscr{I}_{X}$, 
 $\mathscr{I}_{X}^{\ast}$, and the inverse monoid 
$\mathscr{P\!A}\!\left(V\right) $
of partial automorphisms of a vector space $V$ are examples of inverse $\wedge$-semigroups which have significant extra properties, and are important for representations.   Moreover, their local algebras are (isomorphic with) inverse $\wedge$-semigroups of the same kind, to wit partial permutation, block permutation or partial automorphism monoids respectively.   

As usual, $E(A)$ denotes the set of idempotents in $A$. We shall be concerned with stronger order properties of inverse $\wedge$-semigroups, which are often linked with properties of $E(A)$. In the remainder of this section, we give the usual definitions for posets or semilattices in general, but apply them to inverse $\wedge$-semigroups.
Throughout, $S$
is  a subsemigroup (inverse, remember) of $A$, a relationship we notate by $S\leq A$. 
\subsection*{Order properties} 
As is the case for any semilattice, we say $A$ is a \emph{complete} inverse $\wedge$-semigroup if each of its non-empty subsets $X$ has an infimum $\inf X$ in the
natural ordering. In particular, such an $A$ possesses a bottom element and multiplicative zero $0=\inf A$. If $X$ is also bounded above, the supremum $\sup X$ exists.  In particular, if $A$ is \emph{unital}, i.e., has an identity element $1$,  $E(A)$ is  a lattice.  We write $x\vee y$ for $\sup\{x,y\}$ (when it exists).  

As usual, $A$ is \emph{\ distributive} if $x(y\vee z)=xy\vee xz$ for all
$x,y,z\in A$ with $y,z$ bounded above, and
\emph{completely\ distributive} if $x(\sup Y)=\sup\{xy\colon y\in Y\}$
for all $x\in A$ and all $Y\subseteq A$ such that $Y$ has an upper bound in
$A$. 
The following result, known as \emph{Ehresmann's lemma} [Schein; \cite{LePLMS}, section 1.28], is so central to this work  that it can hardly be a fault to include the short proof.

\begin{proposition} \label{ehr}
Let $A$ be a complete inverse $\wedge$-semigroup.   If $X\subseteq A$ and $X$ is bounded above by
$u\in A$, then $X$ has a least upper bound $\sup X$ given by  \[
\sup X=\left(  \sup\{xx^{-1}\colon x\in X\}\right)  u=u\left(\sup\{x^{-1}x\colon x\in X\} \right) .\] 
\end{proposition}
 
\begin{proof}
First, note that $\sup\{xx^{-1}\colon x\in X\}$ exists since $xx^{-1}\leq uu^{-1}$ for all $x\in X$.  Now $x\in X$ implies 
$x = xx^{-1}u \leq  (  \sup\{xx^{-1}\colon x\in X\}) u$, so the latter is an upper bound for $X$. 
But if $b$ is any upper bound, there also hold  $xx^{-1}\leq bu^{-1}$ and so $\sup \{xx^{-1}\colon x\in X\}\leq bu^{-1}$.  Then $(\sup\{xx^{-1}\colon x\in X\}) u \leq bu^{-1}u\leq b$, and  $(\sup\{xx^{-1}\colon x\in X\}) u $ is the least upper bound.  The second equation is dual.
\end{proof}

Recall that a lattice is \emph{atomistic} if every element is a join of atoms; for complete boolean lattices, this is equivalent to being atomic (i.e., every element is above an atom).  
Some background facts concerning atoms will be required.  Let $P = P(A)$ represent the set  
of atoms of $E(A)$.  

\begin{lemma} \label{tfae} 
For $p,q\in P$ and $s\in S$, the following are equivalent:
\begin{enumerate}%[label={(\roman*)}]
 \item $q=s^{-1}ps;$
 \item $ps=sq\neq 0;$
 \item $psq=ps=sq\neq 0;$
 \item $psq\neq 0.$
\end{enumerate}
\end{lemma}

\begin{proof}

(1) $\Rightarrow$ $sq=ss^{-1}ps=ps$ and $ps\neq0$ (else $q=0$) $\Rightarrow$
(2) $\Rightarrow$ $psq=ps\Rightarrow$ 
(3) $\Rightarrow$ (4)$\Rightarrow
s^{-1}psq\neq0\Rightarrow s^{-1}psq=q$ $\Rightarrow$(1). 
\end{proof}

We may simplify proceedings by dividing the problem: let the \rep~ $\rho\colon T\rightarrow A$  factor through the surjection $\psi\colon T\rightarrow S$ and the inclusion $S\hookrightarrow A$.  The structure of $\psi$ is known through the characterisation of congruences on inverse semigroups ({\it via} Preston's kernel-normal systems, or the kernel-and-trace of Scheiblich---see \cite{La} or \cite{Pe}). So we need consider only how a subsemigroup embeds in $A$.  
To avoid needless repetition, we make the following convention:

\emph{For the remainder of this paper, $A$ will denote a unital complete atomistic inverse $\wedge$-semigroup, and $S$ will denote an inverse subsemigroup of $A$.}

\section{A decomposition for semigroups with zero}
The first theorem is simple, but is included mainly for contrast with the classical case, where it is a hidden corollary of the main theorem.  
We begin with a construction, well-known in general semigroup theory\footnote{For instance, it occurs in Vol. II of \cite{C&P}, p.13 as the \emph{$0$-direct union}.}. 
\begin{defs} \label{0ds}
\begin{enumerate}
\item A semigroup $T = T^{0}$ is the \emph{$0$-direct sum} of $0$-disjoint semigroups $T_{\alpha}$ if $T\smallsetminus \{0\} = \bigcup_{\alpha}\{(T_{\alpha}\smallsetminus\{0\})\}$ and for $t\in T_{\alpha}$ and $u\in T_{\beta}$, $T$ has the product \[
t\cdot u = \left\{\begin{array}{ccc}tu & \text{if} & \alpha= \beta ,\\0 & \text{if} & \alpha\neq \beta~.\end{array}\right.
\]
\item Each subsemigroup $T_{\alpha}$ is called a \emph{summand} of $T$, and a semigroup $T$ is \emph{irreducible} if it cannot be written as a $0$-direct sum having more than one non-trivial summand.
\end{enumerate}
\end{defs}
The $0$-direct sum is thus the limit of $0$-preserving injective maps $T^{0}_{\alpha} \hookrightarrow T$ and is written $\sum^{0}T_{\alpha}$. If each $T_{\alpha}$ is inverse, so also is $T$. When $T_{\alpha} = T_{\alpha}^{0}$ for all $\alpha$, as holds for inverse $\wedge$-semigroups, we may write simply $\sum T_{\alpha}$.  Note that the definition implies that when $\alpha\neq \beta$, $T_{\alpha}T_{\beta} = \{0\}$. 

\begin{proposition} \label{sigma}
 Let $T_{\alpha}\leq A$ for $\alpha\in I$.  The map $\sigma\colon t\mapsto \otimes u_{\alpha}$  of $\sum^{0}T_{\alpha}$ to $\prod T_{\alpha}^{0}$, where \[
 u_{\alpha} = \left\{\begin{array}{ccc} t & \text{if} & t\in  T_{\alpha}, \\0 & \text{otherwise,} & \end{array}\right. \]
  is a $0$-preserving injective morphism.%Schein sum.  
\end{proposition}
\begin{proof}
By definition $0\sigma = 0$.  For injectivity, if $t\in \sum^{0}T_{\alpha}$, say $t \in T_{\beta}$, then the set of entries in $t\sigma$ is
$\{0, t\}$, so $\sup _{\alpha}\{t\sigma\} = t$.   
Suppose $s\in T_{\alpha}$ and $t\in T_{\beta}$.  By definition,  $(s\sigma)( t\sigma) =  \otimes u_{\gamma}$ where   \[
u_{\gamma} = \left\{\begin{array}{ccc} st & \text{if} & \alpha = \beta = \gamma ,\\0 & \text{otherwise.} & \end{array}\right.
\]
Thus whether  $\alpha = \beta$ or $\alpha\neq \beta$ (when $T_{\alpha}T_{\beta} = \{0\}$)  we have $\otimes u_{\gamma} = (st)\sigma$ and $\sigma$ is a homomorphism. 
\end{proof}

\begin{defs} \label{e}
Let $S$ be a subsemigroup of $A$. We shall say that 
\begin{enumerate}
\item $S$ is \emph{weakly effective} if the only local algebra
containing $S$ is $A$ itself, i.e., $S\leq eAe$ implies $e=1_{A}$  ($s=se=es$ for all $s \in S$  $\Rightarrow e=1_{A}$);
\item  $S$  is \emph{(strongly) effective} if %$\mathcal{T}_{S}$ is total, i.e.,
there is no $p\in P$ such that $ps=0$ for all $s\in S$.
 \end{enumerate}
Moreover, we shall say a representation $\rho\colon T \rightarrow A$ has either of these properties if its image $T\rho$ has the  corresponding property as a subsemigroup.  
\end{defs}

These properties are equivalent in the classical case, as we discuss in the Appendix, but our first theorem requires the weaker form in general.  For brevity, we shall say that a map $\rho\colon T\rightarrow A$ is a \emph{$0$-representation} if it is a representation such that $0_{T}\rho = 0_{A}$.  (It is not excessively restrictive to consider these, since $\{x\in A\colon x\ge 0_{T}\rho\}$ is a subalgebra of $A$ containing $T\rho$.)
\begin{theorem}\label{th0}
Every $0$-representation  $\rho\colon T\rightarrow A$ of an inverse semigroup $T$ with zero in a unital complete atomistic inverse $\wedge$-semigroup $A$ is a $0$-direct sum of irreducible weakly effective $0$-representations in local algebras of $A$.
\end{theorem}
\begin{proof}
 We begin by writing $S = T\rho$ and 
defining a relation $\mathcal{N}\subseteq P\times S$ as follows:
 \[ (p,s) \in \mathcal{N} \iff ps\neq 0 \text{  or  }  ps^{-1}\neq 0. \]  
 (We may write just $0$ without confusion, since $0_{S} = 0_{A}$ by hypothesis.)  The domain of $\mathcal{N}$ is $Q = \{p\in P\colon ps\neq 0 \text{  for some  } s\in S \} $ and its range is $S^{\ast} = S\setminus\{0\}$.  Note that for $p\in P$, $ps\neq 0$ is equivalent to $p=pss^{-1}$.  Obviously $\mathcal{N}\circ \mathcal{N}^{-1}$ and $\mathcal{N}^{-1}\circ \mathcal{N}$ are symmetric relations, and reflexive on their respective domains $Q$ and $S^{\ast}$.  Next we define their transitive closures,
\begin{align*}
  \mathcal{U} = (\mathcal{N}\circ \mathcal{N}^{-1})^{T} =  \bigcup_{n\in \mathbb{N}}(\mathcal{N}\circ \mathcal{N}^{-1})^{n}  \text{   and   } \\
  \mathcal{K} = (\mathcal{N}^{-1}\circ \mathcal{N})^{T} =  \bigcup_{n\in \mathbb{N}}(\mathcal{N}^{-1}\circ \mathcal{N})^{n} .
\end{align*}
These are equivalence relations on $Q$ and $S^{\ast}$ respectively, and  
it follows from the definition that
$\mathcal{U}\circ\mathcal{N} = \mathcal{N}\circ\mathcal{K}$.  
Thus  the assignment $p\mathcal{U}\mapsto (p\mathcal{N})\mathcal{K}$ is well-defined and has inverse $s\mathcal{K}\mapsto (s\mathcal{N}^{-1})\mathcal{U}$.   
 $\mathcal{U}$- and $\mathcal{K}$- classes (denoted $Q_{\alpha}$ and $S_{\alpha}$, say) have a common indexing such that $p\in Q_{\alpha}$ if and only if there is $s\in S_{\alpha}$ such that $ps \neq 0$ or $sp \neq 0$, and $s\in S_{\alpha}$ if and only if there is $q\in Q_{\alpha}$ such that $qs\neq 0$.  
 
We shall show that, for each $\alpha$, $S_{\alpha}^{0} = S_{\alpha}\cup\{0\}$ is an ideal of $S$.  To this end, suppose $s\in S_{\alpha}$ and $t\in S$.  If $st\neq 0$, there is an atom $c$ such that $c\leq st$.  Put $p = cc^{-1}$; then  $c =cc^{-1}st$, 
and we have $pst, ps \neq 0$. Therefore $s\,\mathcal{N}^{-1}p\,\mathcal{N}st$ and 
$(s,st)\in \mathcal{N}^{-1}\circ \mathcal{N}\subseteq \mathcal{K}$, whence $st\in S_{\alpha}$.  In any case, $st\in S_{\alpha}^{0}$; similarly, $ts\in S_{\alpha}^{0}$ and  $S_{\alpha}^{0}$ is an ideal of $S$.  

It follows that if $\alpha \neq \beta$, $S_{\alpha}^{0}S_{\beta}^{0} \subseteq S_{\alpha}\cup S_{\beta}\cup\{0\}$, and so  
$S = \sum^{0}S_{\alpha}^{0}$.  %, the first claim of the Theorem.

Now set $e_{\alpha} = \sup Q_{\alpha}$ and $A_{\alpha} = e_{\alpha}Ae_{\alpha}$.  Take $s\in S_{\alpha}$ and let $a$ be an atom of $A$ such that $a\leq s$.  With $p$ written for $aa^{-1}$, we have $p\leq ss^{-1}$ and so $ps\neq 0$, whence   $p\in Q_{\alpha}$ and $a=pa\leq e_{\alpha}s$.  Similarly $s = se_{\alpha}$ and moreover $S_{\alpha}^{0}\leq A_{\alpha}$.
    
Suppose that $S_{\alpha} = T_{1}\oplus T_{2}$, for non-trivial subsemigroups $T_{1}$ and $T_{2}$ such that $T_{1}T_{2} = \{0\}$.  If there are $s'\in T_{1}, t'\in T_{2}$ with $(s',t')\in \mathcal{K}$, there must be some $s\in T_{1}, t\in T_{2}$ and $p\in P$ such that $p\leq ss^{-1}$ and $p\leq tt^{-1}$.  But then $p\leq ss^{-1}tt^{-1} = 0$, a contradiction, showing that $S_{\alpha}$ is irreducible.  
If $p\in Q_{\alpha}$, then there is $s\in S_{\alpha}$ such that $p=pss^{-1}\leq ss^{-1}$.  If also $s\in eAe$ for some $e=e^{2}\in A_{\alpha}$, then $s=es$ and so $ss^{-1}\leq e\leq e_{\alpha}$, whence $p\leq e$.  It follows that $e_{\alpha}=\sup Q_{\alpha}\leq e\leq e_{\alpha}$ and so $S_{\alpha}$ is weakly effective in $A_{\alpha}$.  
This establishes the Theorem. 
\end{proof}
Since the structure of irreducible subsemigroups can be quite general, we appear to also need the finer partitioning of the atomic idempotents of $A$ which is provided by the transitivity relation of the next definition. 
%%%%%%%%%%%%%
\section{The orbit equivalence}
\begin{definition} \label{tee}
Define a relation $\mathcal{T}=\mathcal{T}_{S}$ on the set $P$ as follows: for
$p,q\in P,$ $p\mathcal{T}_{S}q$ if there exists $s\in S$ such that
$q=s^{-1}ps$ (or any of the equivalents in Lemma \ref{tfae}).
\end{definition}
Note that $s(s^{-1}ps) = ps$ and that $p\leq ss^{-1}$ iff $ps\neq0$. Thus $\mathcal{T}$ is symmetric ($psq\neq0$ implies $qs^{-1}p\neq0$) and transitive 
($p=s^{-1}qs$ and $q=t^{-1}rt$ imply $p=(ts)^{-1}rts$).  
So in general, $\mathcal{T}$ is only a partial equivalence, that is,
an equivalence on its domain  ${\rm dom}\mathcal{T} = 
\{p\in P\colon ps\neq 0 \text{ for some }s\in S\} = {\rm dom}\,\mathcal{U} = Q$.  It is also easy to see that $\mathcal{T}\subseteq \mathcal{U}$.  

\begin{remark}
The atoms of $A$ which lie beneath some $s\in S$ are precisely the elements of the form $a = psq$ such that $psq\neq 0$ (and then $p,q\in \rm{dom}\mathcal{T}$).  As $s$ ranges over $S$, these elements form a groupoid,  the connected components of which are the $\mathcal{T}$-classes. 
\end{remark}
 \begin{defs}
 Let the $\mathcal{T }$-classes into which $\rm{dom}\mathcal{T}$ is partitioned be indexed by the set $I$, and denoted by $\{P_{i}\colon i\in I\}$.  Define (for each $i\in I$) the idempotent $e_{i} = \sup P_{i}$ and the local algebra $A_{i} = e_{i}Ae_{i}$.
\end{defs}
It is evident that $\mathcal{T}_{S} \subseteq \mathscr{D}^{A}$ for any subsemigroup $S$,
 so a $\mathcal{T}_{S}$-class $P_{i}$ must be contained within a single $\mathscr{D}$-class of $A$.  
 Before leaving this section we record a useful result.
\begin{lemma}\label{ps}
Let $P'\subseteq P$ and $s\in S$.  Then 
$$\sup \{ps\colon p\in P'\} = (\sup\{p\colon p\in P', ps\neq 0\})s.$$ 
\end{lemma}
\begin{proof}
 It is elementary to prove that, if $X, Y$ be subsets of a complete semilattice such that $X\cup Y$ is bounded above, 
$\sup(X\cup Y) = (\sup X)\vee  (\sup Y)$.  
  Then, 
observing $\{ps\colon p\in P'\} = \{ps\colon p\in P', ps\neq 0\}\cup \{ps\colon p\in P', ps = 0\}$, we have  
\begin{align*} 
\sup\{ps\colon p\in P'\} &= \sup\{ps\colon p\in P', ps\neq 0\}\vee \sup\{ps\colon p\in P', ps = 0\}\\ &= \sup\{ps\colon p\in P', ps\neq 0\}. 
\end{align*}
Now  $ \{ps\colon p\in P'\}$ has $s$ as an upper bound, and by Proposition \ref{ehr} and Lemma \ref{tfae}, 
\begin{align*} 
 \sup \{ps\colon p\in P'\} &= (\sup\{ps(ps)^{-1}\colon p\in P', ps\neq 0\})s
&= (\sup\{p\colon p\in P', ps\neq 0\})s .
\end{align*}
\end{proof}
 
\section{Transitivity; bounded sums}
In the classical theory, effectiveness (Definition \ref{e}) and transitivity (below) are key properties of representations and of subsemigroups. 
\begin{defs} \label{t}
Let $S$ be a subsemigroup of $A$. We shall say that 
\begin{enumerate}
\item $S$ is \emph{weakly transitive} if
$\mathcal{T}_{S}$ has just one class, that is, for each pair $p,q\in P$ such
that $pS\not =\left\{  0\right\}  $ and $qS\not \neq \left\{  0\right\}  ,~$
$p=s^{-1}qs$ for some $s\in S$;
\item $S$ is \emph{(strongly) transitive} in
$A$ if $\mathcal{T}_{S}$ is the universal relation on $P$, i.e., for all $p,q\in P$, there is some $s\in S$ such that $psq\neq 0$.
 \end{enumerate}
As before, we  say a representation $\rho\colon T \rightarrow A$ has either of these properties if its image $T\rho$ has the  corresponding property as a subsemigroup.  
\end{defs}
If $S$ is weakly transitive, it is irreducible (Definition \ref{0ds}(2)); for $psq\neq 0$ implies $ps,sq\neq 0$ so $(p,q)\in \mathcal{U}$.
Strong transitivity has implications for the structure of $A$:
$S$ is  transitive if, and only if, for each pair $p,q\in P$ there
exists $a\in A$ such that $p=a^{-1}a,\,q=aa^{-1}$, and $a\leq s$ for some
$s\in S$; that is, the $\mathscr{H}$-class $R_{p}\cap L_{q}$ contains an
element beneath some element of $S$. In particular, all atoms of $A$ then form one
$\mathscr{D}$-class.

The Appendix to the present paper contains a discussion of the rationale for the choices of the generalisations made here, and includes small examples.
The reader will note that each generalised property possesses a weak and a strong version; after making the definitions above, we shall generally suppress the modifier `strong', etc.  
This allows a simplified terminology, explained in the following Lemma.
\begin{lemma} \label{wis}
 If the subsemigroup $S$ is transitive then it is effective; if it is effective and  weakly transitive, then it is transitive.  
\end{lemma}
\begin{proof}
If $S$ is transitive, and $p \in P(A)$, then there is $s \in S$ with $sp \neq 0$ and $S$ is effective.
Let $ S $ be effective, and $p,q \in P(A)$.  
Then there exist $s,t \in S$ such that $sp,tq\neq 0$, and in turn this means that $p,q \in \rm{dom}(\mathcal{T}_{S})$. With weak transitivity this implies $p \mathcal{T}_{S} q$.  
\end{proof}
Thus we may use the modifier `weakly' to refer to both attributes in conjunction, so `weakly effective and transitive' is to be read as `weakly effective and weakly transitive'.
 
Analysis of the classical case suggests the need for yet another construction.
\begin{defs}\label{otimes}
Consider a collection of semigroups $T_{i}$ indexed by $i\in I$ and having the product $\prod \{T_{i}\colon i\in I\}$, or briefly $\prod T_{i}$.
\begin{enumerate}[label=({\roman*})]
\item
 We write $\otimes x_{i}$ to denote the `sequence'  $(x_{i})_{i\in I}$, i.e., the member of \,$\prod T_{i}$ such that $(\otimes x_{i})pr_{j} = x_{j}\in T_{j}$.  
\item
Given maps  $\psi_{i}\colon U\rightarrow T_{i}$, the unique map $U \rightarrow \prod{T_{i}}$ provided by the limit property is called the \emph{product} of the maps $\psi_{i}$ and denoted by $\otimes\psi_{i}$.  It satisfies $s(\otimes\psi_{i}) = \otimes (s\psi_{i})$ and is a homomorphism if each $\psi_{i}$ is.  
\item
 Now suppose that each $T_{i}\leq A$.  An element $t = \otimes t_{i}$ of \;$\prod T_{i}\leq A^{I}$ will be called \emph{bounded} if the set $\{t_{i}\}$ is bounded above in $A$.  The set of all bounded elements of $\prod T_{i}$ will be denoted $B = \mathscr{B}(\prod T_{i})$.
\end {enumerate}
%\end{definition}
\end{defs}

Let us next observe that $B$ thus defined is the (maximum) domain of the partial function 
$\sup\colon \otimes t_{i}\mapsto \sup\{t_{i}\}$, and  is a subsemigroup of $\prod T_{i}$.  For if $s=\otimes s_{i}, t=\otimes t_{i} \in B$, then for all $i\in I$ there hold $s_{i}\leq u, t_{i}\leq v$ for some $u,v\in A$, and so $s_{i}t_{i}\leq uv\in A$, whence $st\in B$; and also $s_{i}^{-1} \leq u^{-1}$, whence $s^{-1}\in B$.  Indeed $B$ is an inverse $\wedge$-semigroup in its own right.
\begin{defs} \label{Sprod}
\begin{enumerate}[label=({\roman*})]
\item Define $\omega\colon \mathscr{B}(\prod T_{i})\rightarrow A$ by 
$(\otimes t_{i})\omega = \sup\{t_{i}\}$ (useful when we wish to write $\sup$ as a right mapping).
\item
Any subsemigroup of $B$ will be called a \emph{bounded} subsemigroup of $\prod T_{i}$, and a \emph{Schein sum} if $\omega$ is a homomorphism.  
\item
Given maps  $\psi_{i}\colon S\rightarrow T_{i}$, the map  $\otimes\psi_{i}\colon S \rightarrow \prod{T_{i}}$ will be called a \emph{bounded product} if its image is bounded, and a \emph{Schein sum} if its image is Schein.  
\item
The representations $\phi \colon S \rightarrow A_{1}$ and $\psi\colon S \rightarrow A_{2}$ are \emph{[weakly] equivalent}
if there is an isomorphism $\theta\colon A_{1}\rightarrow A_{2}$ \emph{ [$\theta\colon S\phi\rightarrow S\psi $]} such that $\phi\theta= \psi$.
%\end{definition}
\end{enumerate}
\end{defs}
It should be noted that the $0$-direct sum of Theorem \ref{th0} is a Schein sum by Proposition \ref{sigma}.  

\section{Representation by Schein sums} \label{S}
With the apparatus of the previous section at hand, we may formulate the second theorem.
\begin{theorem}\label{th1}
 Every %{\rm [}effective{\rm ]} 
 representation of an inverse semigroup $T$ in a complete atomistic inverse $\wedge$-semigroup $A$ is weakly equivalent to a Schein sum of weakly %transitive {\rm [}transitive{\rm ]}
 effective representations in local algebras of $A$. %which is unique up to order of the factors.
\end{theorem}
\begin{proof} Let $S = T\rho$.
Fix an element $s \in S$, and for each $i\in I$, set 
$s_{i} = \sup\{ps\colon p\in P_{i}\}$.
By Lemma \ref{ps}, $s_{i} =  (\sup\{p\colon p\in P_{i}, ps\neq 0\})s = e_{i}s$, and dually $s_{i} = se_{i}$.
It follows that $s\phi_{i} = s_{i}$ defines a map $\phi_{i} \colon S \rightarrow e_{i}Ae_{i}$, with $A_{i} = e_{i}Ae_{i}$ being a local  algebra of $A$.  
In fact, $\phi_{i}$ is a homomorphism, since 
$(st)\phi_{i} = e_{i}(st)e_{i} = (e_{i}s)(te_{i}) = (s\phi_{i})(t\phi_{i}).$ We denote the subsemigroup $S\phi_{i}$ by $S_{i}$.

Next we show that for each $s \in S$, $s = \sup\{s\phi_{i}\colon i\in I\} $.
 Clearly $ \sup\{s\phi_{i}\colon i\in I\} \leq s$.  For the reverse inequality, let $a\in A$ be an atom such that $a\leq s$ and let $q= aa^{-1}$.  Then $a = qs$  and $q \in \rm{dom}\mathcal{T}$, so there is $i \in I$ such that $q \in P_{i}$.  Thus $a\leq e_{i} s =s\phi_{i} \leq \sup \{s\phi_{i}\colon i\in I\}$, so $s \leq \sup\{s\phi_{i}\colon i\in I\} $ and equality follows. 
 
Then with  $\otimes \phi_{i}$ as defined in Definitions \ref{otimes}, $s(\otimes\! \phi_{i})\omega = \sup\{s_{i}\} = \sup\{se_{i}\} = s$ for each $s\in S$. Thus $\omega$ is an injective morphism 
and   $\rho\otimes\phi_{i}$ is a Schein sum map by Definition \ref{Sprod}.  Moreover, $\rho\otimes\phi_{i}$ is weakly equivalent to $\rho$.
  
 Let $S_{i}\subseteq eA_{i}e$ and $p\in P_{i}$.  Then there is $s$ with $es=s=se$ such that $p = pss^{-1}$, and $ss^{-1}\leq e$.  Thus $p\leq e$ and  $e_{i}\leq e$ follows. So $S_{i}$ is weakly effective  by Definition \ref{e}(1)  and the proof is complete.
\end{proof}
%\end{figure}
Note our lazy re-use of the $S_{i}$ symbol: these need not be the same as the $S_{\alpha}$ of Theorem \ref{th0} (indeed several $i$ may correspond to a single $\alpha$, since $\mathcal{T}\subseteq\mathcal{U}$).  Despite the finer partitioning of $P$ here, the method of Theorem \ref{th0} may offer an advantage in allowing the partitioning of $S$. 
Also standing in contrast with Theorem \ref{th0}, the decomposition $\rho\otimes\phi_{i}$ may be quite distinct from $\rho$, even when $\rho$ is an embedding.  See Example (1) of the Appendix for a simple illustration of these points.

Refinements of Theorem \ref{th1} may be obtained when the original representation has additional properties, which we proceed to discuss.
 
\section{Dispersed representations}
Here we introduce a condition which ensures the components are at least weakly transitive. 
\begin{lemma} \label{Tres1}
 For each $i$, $\mathcal{T}_{S_{i}} = \mathcal{T}_{S}\cap (A_{i}\times A_{i})$.
 \end{lemma}
\begin{proof}
If $(p,q)\in \mathcal{T}_{S_{i}}$, there is $s_{i}=se_{i}$ such that $ps_{i} = pe_{i}s = se_{i}q = s_{i}q \neq 0$. Since $0\neq pe_{i}\leq p$, $pe_{i}=p$ and $e_{i}q = q$ similarly.  Thus $p, q \in A_{i}$; moreover, substitution shows $ps = sq\neq 0$, whence        $(p,q)\in \mathcal{T}_{S}\cap (A_{i}\times A_{i})$.
 For the reverse inclusion,  $(p,q)\in \mathcal{T}_{S}$ and $p,q \in A_{i}$ imply $p=pe_{i}, q=e_{i}q$, and $ps=sq\neq 0$. Thus $pe_{i}s=e_{i}sq\neq 0$, ie., $ps_{i}=s_{i}q\neq 0$ with $s_{i}\in S_{i}$.
\end{proof}

\begin{lemma}\label{refine}
 If $A_{i}\cap P_{j}\neq \varnothing$ then $P_{j}\subseteq A_{i}$.
\end{lemma}
\begin{proof}
 Let $p\in A_{i}\cap P_{j}$ and $q\in P_{j}$. Since $(p,q)\in \mathcal{T}$, we have (by Lemma \ref{tfae}) that
 $q= s^{-1}ps = s^{-1}e_{i}pe_{i}s = s_{i}^{-1}ps_{i} = e_{i}qe_{i} . $%
 \end{proof}
\begin{lemma} \label{propsSi} For each $i$,
\begin{enumerate}[label=({\roman*})]
\item $S_{i}$ is  effective in $A_{i}$ if and only if $P\cap A_{i}$ is a union of classes $P_{j}$; 
\item $S_{i}$ is  weakly transitive in $A_{i}$ if and only if $P_{j}\subseteq A_{i}$ implies $i=j$;
\item $S_{i}$ is  transitive in $A_{i}$ if and only if $P\cap A_{i} = P_{i}$. \end{enumerate}
\end{lemma}
\begin{proof}
(i) $P\cap A_{i}$ is a union of classes $P_{j}$ if and only if $\mathcal{T}_{S_{i}}$ is total on $P\cap A_{i}$, so the claim follows from  Definition \ref{e}.

(ii, iii)  By Lemma \ref{Tres1}, $\mathcal{T}_{S_{i}}$-classes are precisely of the form $A_{i}\cap P_{j}$, so the statements are equivalent to the respective Definitions \ref{t}.   
\end{proof}
\begin{defs}\label{dispersed}
 A subsemigroup $S$ of $A$ is \emph{dispersed} if $A_{i}\cap P_{j}\neq\varnothing$ implies $i=j$.  A representation $\rho\colon T \rightarrow A$ is \emph{dispersed} if its image $T\rho$ is a dispersed subsemigroup of $A$.
\end{defs}
\begin{theorem}\label{th2}
Every dispersed representation $\rho$ of an inverse semigroup $T$ in a complete atomistic inverse $\wedge$-semigroup $A$ is weakly equivalent to a Schein sum of weakly effective and transitive representations in local algebras of $A$.  
If $\rho$ is effective, each factor is transitive. 
\end{theorem}
\begin{proof}
Theorem \ref{th1} applies, and in addition, Definition \ref{dispersed} and Lemma \ref{propsSi}(iii) tell us that each factor $S_{i}$ is weakly transitive.  If $\rho$ is effective, Lemma \ref{propsSi}(iv) gives transitivity of $S_{i}$.
\end{proof}
Note we have not claimed even essential uniqueness; item (4) of the Appendix shows why. 
\section{The distributive case}
In the classical case, the local algebras $A_{i}$ are $0$-disjoint and so also the $S_{i}$. The property of $\mathscr{I}_{X}$ which chiefly brings this about is (complete) distributivity.  Complete distributivity of $A$ implies that it is (unital) Boolean as defined in \cite{La12}: indeed $e\in E(A)$ has complement defined by 
$\bar{e} = \sup\{p\in P\colon pe = 0\}$.  
%%%%%%%%%%%%%%%%
\begin{definition} \label{orthog}
A Schein sum $\otimes \psi_{i}\colon S\rightarrow \prod T_{i}$ is called \emph{orthogonal} if for all $i\neq j$, $(S\psi_{i})(S\psi_{j}) = \{0\}$.
\end{definition}

\begin{theorem}\label{th3}
 Every effective representation of an inverse semigroup $T$ in a completely distributive atomistic inverse $\wedge$-semigroup $A$ is 
 an orthogonal Schein sum of transitive representations in local algebras of $A$, which is unique up to order of the factors.
\end{theorem}
\begin{proof}
We continue using the notation introduced above.  
Suppose that $p\in P_{j}$ and $p\leq e_{i} = \sup P_{i}$.  Then $p = p\sup\{q\colon q\in P_{i}\} = \sup\{pq\colon q\in P_{i}\}$, since $A$ is distributive.
However $pq\neq 0$ if and only if $p = q$ for some $q\in P_{i}$, and it follows that $p\in P_{i}$. Thus $S$ is dispersed and Theorem \ref{th2} applies.  

Suppose $p\in P$ satisfies $p\leq e_{i}e_{j}$.  Then $p\in A_{i}\cap A_{j}\cap P = P_{i}\cap P_{j}$ by hypothesis,   so either $i = j$ or there is no such $p$, and in this case  $e_{i}e_{j} = 0$.  So if $i\neq j$, $S_{i}\cap S_{j} \subseteq A_{i}A_{j} = \{0\}$ and 
the Schein sum is orthogonal (Definition \ref{orthog}).

Since $S$ is dispersed and effective, we have $P\cap A_{i} = P_{i}$ and transitivity of $S_{i}$ follows from Lemma \ref{propsSi}(iii).  
As to essential uniqueness: 
suppose that also $S \leq \mathscr{B}(\prod\{U_{j}\colon j\in J\})$ where (for each $j\in J$) $U_{j}$ is weakly transitive in $B_{j}$, a local algebra of $A$, say $B_{j} = f_{j}Af_{j}$.  Let $(p,q)\in \mathcal{T}_{S}$, so $psq \neq 0$ for some $s \in S$.  We have 
$s = \sup\{u_{j}\colon u_{j}\in U_{j}\}$, so $p(\sup\{u_{j}\})q = \sup\{pu_{j}q\}\neq 0$.  
Thus there is $j_{0}\in J$ such that $pu_{j_{0}}q\neq 0$, whence $(p,q)\in \mathcal{T}_{U_{j_{0}}}\subseteq \mathcal{T}_{S}$.  In particular, $p,q\in P(B_{j_{0}})$.  Conversely, because $U_{j_{0}}$ is transitive in $B_{j_{0}}$, 
$p,q\in P(B_{j_{0}})$ implies $(p,q)\in \mathcal{T}_{U_{j_{0}}}$ and so $p,q \in P_{i_{0}}$ for some (unique) $i_{0}$. 
It follows that $f_{j_{0}} = \sup P(U_{j_{0}}) = e_{i_{0}}$, and in turn that $B_{j_{0}} = A_{i_{0}}$ and that 
$\phi_{i_{0}}\colon s \mapsto f_{j_{0}}sf_{j_{0}} = u_{j_{0}}$.  So the associated representations have the same set of factors, in perhaps different orders.  
\end{proof}
Theorem \ref{th3} specialises to Schein's decomposition result for representations in $\mathscr{I}_{X}$, as expressed in Section VI.4 of Petrich \cite{Pe}.

\begin{section} {Dedication and Acknowledgements}
The paper is dedicated to Vivienne Luke whose support of this and other projects has been more than love requires.  Colleagues gave me opportunities to speak on this topic in various places, albeit without the definitive results above.  Long ago, Jonathan Leech and Victor Maltcev wrote helpful comments on an early draft, and most recently an anonymous referee's wise advice hugely improved the presentation.  It is truly a pleasure to thank them all.
\end{section}

\section{Appendix: Remarks and examples} 
\rm
\small
First, we sketch an elementary proof of the claim (in Section 1) that $\text{deg}S\leq 2^{\text{deg}^{\ast}S}-2$.
We begin by constructing a faithful representation  of $\mathscr{I}^{\ast}_{X}$  in $\mathscr{I}_{2^{X}}$. This is actually a consequence of Theorem 1.5 of \cite{FiLe} applied to the contravariant power set functor on sets, but the concrete details are of interest too.  
Let $\beta \in \mathscr{I}^{\ast}_{X}$, say \[
\beta = \left(\begin{array}{c}\cdots\\ \cdots \end{array}\right| \left. \begin{array}{c}D_{i}\\ R_{i} \end{array} \right| \left.\begin{array}{c}\cdots\\ \cdots \end{array} \right), 
\]
where we use the two-line notation of \cite{FiLe}, Section 2.  There is a corresponding partial permutation $\widehat{\beta}$ of %non-empty 
subsets of $X$, in which $\rm{dom}\widehat{\beta}$ consists of all unions of blocks of $\beta \beta^{-1} = \{D_{i}\colon i\in I\}$, $\rm{ran}\widehat{\beta}$ consists of all unions of blocks of $\beta^{-1} \beta = \{R_{i}\colon i\in I\}$ and, for any  $J \subseteq I$,
\[
(\bigcup\{D_{j}\colon j\in J\}) \widehat{\beta} = \bigcup\{R_{j}\colon j\in J\}. 
\]
 Clearly $\beta \mapsto \widehat{\beta}$ is injective (consider the action on singleton unions) and calculation shows that $\widehat{\beta_{1}\beta_{2} } = \widehat{\beta_{1}}\widehat{\beta_{2}}$ (noting that $\rm{ran}\widehat{\beta_{1}}\cap\rm{dom}\widehat{\beta_{2}}$ consists precisely of the unions of blocks of the partition which is the partition-join of $\rm{ran}\beta_{1}$ with $\rm{dom}\beta_{2}$, so the respective composites correspond).  

Now this map $\beta \mapsto \widehat{\beta}$  always preserves the empty union and the total union ($X$), and so there is a homomorphism, still injective, of $\mathscr{I}^{\ast}_{X}$ to $\mathscr{I}_{Y}$ where $Y=2^{X}\smallsetminus \{ \varnothing, X \}$.  (That this is actually best possible follows from Schein's work \cite{Sc0, Sc1}, or more elementarily by counting the singletons of $Y$ (idempotent atoms of $\mathscr{I}_{Y}$) required to faithfully represent the maximal subgroups in the bottom $\mathscr{D}$-class of $\mathscr{I}^{\ast}_{X}$.)

Now choose $X$ so that $\text{deg}^{\ast}S = \vert X\vert$, so $S$ embeds in $\mathscr{I}^{\ast}_{X}$ and so in $\mathscr{I}_{Y}$, whence  $\text{deg}S\leq \vert Y\vert = 2^{|X|}-2$ and the claim follows.

The second task in this Appendix is to justify the choices presented in Definitions \ref{e} and  \ref{t}. 
Schein, in the context of the
semigroup $\mathscr{B}_{X}$ of binary relations, says that a subsemigroup $S$
is transitive if, given any $x,y\in X,$ there is $s\in S$ with $\left(
x,y\right)  \in s.$ This definition carries over perfectly well to $\mathscr{I}_{X}$, where the most productive view is to consider the action of $S$ on $X$, which is in one-to-one correspondence with the set of idempotent atoms ($\{(x,x)\colon x\in X\}$)---{\it cf.} Petrich \cite{Pe}, where $S$ is transitive [effective] if the relation of transitivity $\mathcal{T}_{S}$ is universal [has total projections].  
Underlying Definition \ref{t}, then, is  an action $(p,s)\mapsto s^{-1}ps$ on the set $X = P\cup\{0\}$, equivalently a representation of $S$ in the transformation semigroup $\mathscr{T}_{X}$.  It restricts to a \emph{partial} action of $S$ on $P = X\setminus \{0\}$ which gives a representation (not necessarily faithful) of $S$ in $\mathscr{I}_{P}$.  

In the classical case, this action is the defining action.  This is not at all the case in general: equivalent classical conditions bifurcate into weak and strong versions, hence the Definitions \ref{t}.   Still, it seems that one should continue to use the  idempotent atoms in these definitions.  In the case of $\mathscr{I}^{\ast}_{X}$ these are dichotomies in $X$, and so atoms of the partition lattice on $X$ \emph{when it is ordered the right way up}---see Ellerman \cite{Ell, Ell14}.

What about the classical precedent for effectiveness?  
It is easier to first describe ineffective subsemigroups. In
$\mathscr{I}_{X}$, a subsemigroup $S$ is ineffective if (a) there exists a
proper local algebra $\mathscr{I}_{Y}$ containing $S;$ equivalently if (b)
there is at least one idempotent atom $\left\{  \left(  x,x\right)
\right\}  $ such that $x$ is in the domain of no member of $S,$ that is,
$\left\{  \left(  x,x\right)  \right\}S =\varnothing,$ which is the zero of
$\mathscr{I}_{X}.$

Now in the general case, if $A$ is atomistic, (a) implies (b): if (a) holds,
there exist $e\neq1$ with $S\subseteq eAe,$ and $p\in P$ with $p\not \leq e$
(otherwise, $e=\sup P=1$). Thus $pe=0,$ but then $ps=pes=0$ for all $s\in S.$ But the reverse is not true, as we now illustrate. 
\subsection*{Examples}
The following (``non-classical'') examples are chosen to occur in dual symmetric inverse $\wedge$-semigroups $\mathscr{I}^{\ast}_{X}$ of small degree, and we continue to use the two-line notation as before, with the abbreviations 
 $\nabla$ for the zero of $\mathscr{I}^{\ast}_{X}$ (the universal relation or partition on $X$) and $\Delta$ for the identity (the identity relation or partition on $X$).
\begin{enumerate}%[(i)]
\item
Consider a semigroup $T$ which is a $0$-direct sum of a $5$-element aperiodic Brandt semigroup with a $2$-element semilattice. It may be embedded in $\mathscr{I}^{\ast}_{4}$ as the subsemigroup $S = \{\nabla , \delta ,\alpha , \alpha^{-1}, \alpha \alpha^{-1}, \alpha^{-1}\alpha \} = \langle \alpha, \delta\rangle$, where   \[
\alpha  = \left(\begin{array}{c}12 \\13\end{array}\right|\left.\begin{array}{c}34 \\24\end{array}\right) \text{   and   }
\delta   = \left(\begin{array}{c}1 \\1\end{array}\right|\left.\begin{array}{c}234 \\234\end{array}\right) .
\]
The idempotents of $S$ are $\alpha \alpha^{-1} = (12 \vert 34)$, $\alpha^{-1}\alpha = (13\vert 24)$,  and $\delta = (1\vert 234)$  which are all members of $P$.  Checking condition (b) applied to the other members of $P$, note  \[%$\Delta  \triangle
\ (2\vert 134) \nabla = (2\vert 134)\delta  = (2\vert 134)\alpha  = (2\vert 134)\alpha^{-1} = \nabla , 
\]
so that $(2\vert 134)S_{1} = \{\nabla\}$ and (b) is satisfied, so $S$ is ineffective in 
the sense of (b). But condition (a) above is not satisfied: the only local 
algebra containing $S_{1}$  is $\mathscr{I}^{\ast}_{4}$  itself, because the l.u.b of $\alpha \alpha^{-1}$ and $\alpha^{-1}\alpha$ is $\Delta$.
So $S_{1}$ is weakly effective, but not strongly effective. There are two $\mathcal{T}$-classes or orbits, $P_{1} = \{\alpha \alpha^{-1}, \alpha^{-1}\alpha\} = \{(12\vert 34), (13\vert 24)\}$ and $P_{2} = \{\delta \}  = \{(1\vert 234)\}$.
Thus the local identities $e_{i} = \sup P_{i}$ are \[
e_{1} %= \sup P_{1} 
= \alpha\alpha^{-1}\vee\alpha^{-1}\alpha = (12\vert 34)\cap (13\vert 24) = \Delta \hspace{12pt} 
 \text{and} \hspace{12pt}
 e_{2}  = \delta = (1\vert 234);\]
 note that $e_{2}\leq e_{1}$. 
The projection maps $\phi_{i}$ are $\phi_{1} = \text{id}$ and \[
\phi_{2} = \left(\begin{array}{cccccc}\nabla & \delta & \alpha & \alpha^{-1} & \alpha\alpha^{-1} & \alpha^{-1} \alpha\\ \nabla & \delta & \nabla & \nabla & \nabla & \nabla \end{array}\right),  
\]
and their images are $S_{1}=S$ and $S_{2} = \{\delta,\nabla\}\leq S_{1}$;   by Lemma \ref{propsSi}, $S_{1}$ is not weakly transitive and $S_{2}$ is transitive.  The representation $\phi_{1}\otimes\phi_{2}$ has codomain 
$A_{1}\times A_{2} = \mathscr{I}^{\ast}_{4}\times \mathscr{I}^{\ast}_{2}$ and is determined by the images $\alpha (\phi_{1}\otimes\phi_{2}) = (\alpha, \nabla)$ and $\delta (\phi_{1}\otimes\phi_{2}) = (\delta,\delta)$.  
In contrast the representation of $S$ by the method of Theorem \ref{th0} is given by $\alpha\mapsto (\alpha, \nabla)$ and $\delta \mapsto (\nabla, \delta)$.   
\item
We can also embed $T$ in $\mathscr{I}^{\ast}_{4}$ as the subsemigroup 
\[
S = \{\nabla , \delta ,\beta , \beta^{-1}, \beta \beta^{-1}, \beta^{-1}\beta \} , \text{  where }
\beta  = \left(\begin{array}{c}12 \\2\end{array}\right|\left.\begin{array}{c}34 \\134\end{array}\right) \]
and $\delta$ is as before. This time, (a) is satisfied since $S$ is contained in the local algebra whose identity is $(1\vert 2\vert 34)$ and so (b) is satisfied too; $S$ is ineffective on either criterion. In fact $p = (4\vert 123)$ has $pS = \{\nabla\}$. 
\item
But we can modify example (ii) to embed $T$ in $\mathscr{I}^{\ast}_{3}$ by lumping vertices $3, 4$ together to make a 3-element set $\{1,2,3\}$, the 
quotient of $\{1,2,3,4\}$ by the equivalence generated by $(3, 4)$. Thus we consider
%\[
$S = \langle\epsilon, \gamma\rangle = \{\nabla , \epsilon, \gamma , \gamma^{-1}, \gamma \gamma^{-1}, \gamma^{-1}\gamma \}$,  where  %\]
\[\gamma  = \left(\begin{array}{c}12 \\2\end{array}\right|\left.\begin{array}{c}3 \\13\end{array}\right) \text{   and   }
\epsilon   = \left(\begin{array}{c}1 \\1\end{array}\right|\left.\begin{array}{c}23 \\23\end{array}\right) .
\]
Now all the  idempotent atoms in $\mathscr{I}^{\ast}_{3}$ occur already in $S$, whence $p \in pS$ for all $p \in P$, and so this $S$ is effective. The two orbits %$\mathcal{T}$-classes 
are $P_{1} = \{(12\vert 3), (2\vert 13) \}$ and 
$P_{2} = \{(1\vert 23)\}$; the local identities are $e_{1} = (12\vert 3)\cap (2\vert 13) = \Delta$ and $e_{2} = (1\vert 23)$, and the maps as before, $\phi_{1} = \text{id}$ and \[
 s\phi_{2} = \left\{\begin{array}{ccc}\delta , & \text{if} & s = \delta \\\nabla ,& \text{if }& s\neq \delta~ .\end{array}\right.
\]
In each example (1)--(3), the local algebra generated by $P_{1}$ contains the local algebra generated by $P_{2}$; similar examples could be given for subsemigroups in (say) the inverse semigroup of partial automorphisms of a vector space. This contrasts with the classical theory, where the local algebras generated by distinct orbits intersect in $\{\varnothing\}$, the trivial local algebra.  
\item
Last, an  example of a dispersed representation, which incidentally also makes the point that effectiveness is distinct from efficiency (in the sense of a lack of redundancy).  
 
Let $T$ be the $5$-element aperiodic Brandt semigroup generated (as an inverse semigroup with zero) by $a$, subject to the relation $a^{2} = 0$.  It may be embedded in $\mathscr{I}^{\ast}_{5}$ by the map induced by $a \mapsto \alpha = 
 \left(   \begin{array}{c}1 \\2\end{array}\right| \left. \begin{array}{c}4 \\3\end{array}\right|\left. \begin{array}{c}235 \\145\end{array}   \right) $, with $S = \langle \alpha \rangle$.  
 We can calculate the orbits  $P_{1} = \{(1 \vert 2345), (2\vert 1345)\}$, $P_{2} = \{(3\vert 1245), (4\vert 1235)\}$, and $P_{3} = \{(14\vert 235), (23\vert 145)\}$.  From these we have $e_{1} = (1\vert 2\vert 345)$, $e_{2} = (3\vert 4\vert 125)$, and $e_{3} = (14\vert 23\vert5)$.  Note $e_{1}e_{2} = e_{1}e_{3} = e_{2}e_{3} = \nabla$ holds,  
 and $S$ is dispersed (Definition \ref{dispersed}).
 Then 
 $$ \alpha\phi_{1} =   
 \left(   \begin{array}{c}1 \\2\end{array}\right| \left. \begin{array}{c}2345 \\1345\end{array}\right), 
  \alpha\phi_{2} =   
 \left(   \begin{array}{c}4 \\3\end{array}\right| \left. \begin{array}{c}1235 \\1245\end{array}\right), \text{  and }
  \alpha\phi_{3} =   
 \left(   \begin{array}{c}1 4\\2 3\end{array}\right| \left. \begin{array}{c}235 \\145\end{array}\right) .$$
 So $\alpha = \alpha(\phi_{1}\oplus\phi_{2}\oplus\phi_{3})\omega = \alpha(\phi_{1}\oplus\phi_{2})\omega = \alpha(\phi_{1}\oplus\phi_{3})\omega = \alpha(\phi_{2}\oplus\phi_{3})\omega $, and even this dispersed example does not have a unique Schein sum representation. 
 \end{enumerate}

\vspace{0.5cm} 
\end{document}